\documentclass[12pt]{amsart}
\usepackage[margin=1in]{geometry}
\usepackage{amsmath,amssymb,amsthm,mathtools}
\usepackage{hyperref}

\newtheorem{theorem}{Theorem}[section]
\newtheorem{proposition}[theorem]{Proposition}
\newtheorem{corollary}[theorem]{Corollary}
\newtheorem{lemma}[theorem]{Lemma}
\theoremstyle{remark}
\newtheorem{remark}[theorem]{Remark}

\newcommand{\Z}{\mathbb Z}
\newcommand{\Q}{\mathbb Q}
\newcommand{\CP}{\mathbb {CP}}
\newcommand{\Int}{\operatorname{int}}

\title[Torus surgery and the Cartwright--Steger surface]
{Torus Surgery Constraints and Intermediate Covers of the Cartwright--Steger Surface}
\author{Anar Akhmedov}
\address{School of Mathematics, University of Minnesota, Minneapolis, MN 55455, USA}
\email{akhmedov@umn.edu}
\author{Sai-Kee Yeung}
\address{Department of Mathematics, Purdue University, West Lafayette, IN 47907, USA}
\email{yeungs@purdue.edu}
\date{}

\begin{document}
\begin{abstract}
We study torus surgery on the Cartwright--Steger surface and the degree-two cover construction proposed in \cite[Remark~1]{ASY}.  Every torus in this surface is integrally nullhomologous, and surgery on a fixed disjoint collection of such tori cannot decrease the first Betti number.  We then determine the nodal-curve subgroup and complete the fundamental-group calculation in the degree-two construction.  The resulting minimal symplectic manifold has signature zero and fundamental group $\Z/2$; for compatible symplectic gluings, its universal cover is an exotic symplectic $31\CP^2\#31\overline{\CP}^{\,2}$.  The construction on the original Cartwright--Steger surface gives a minimal symplectic manifold with signature $-1$ and fundamental group $(\Z/2)^2$.
\end{abstract}

\maketitle

\section{Introduction}

The Cartwright--Steger surface lies at a useful meeting point of complex
hyperbolic geometry and symplectic four-manifold topology.  As a compact ball
quotient it is strongly constrained by negative curvature, while its irregular
topology and the curves constructed in \cite{CKY} make it a natural building
block for surgery and symplectic sum.  This raises two related questions.  How
far can torus surgery change its topology, and what remains of its fundamental
group when one uses its nodal curve in a symplectic sum?

Let $X$ denote the Cartwright--Steger surface.  Its numerical invariants are
\[
e(X)=3,\qquad \sigma(X)=1,\qquad H_1(X;\Z)\cong\Z^2.
\]
The surface is a compact quotient of the complex two-ball \cite{CKY} and is therefore aspherical and admits a metric of strictly negative real sectional curvature.

The first question has a definitive answer for surgery on a fixed disjoint
torus link.  Negative curvature and asphericity imply that every torus in $X$
is integrally nullhomologous.  A relative homology calculation then shows that
surgery on such a link cannot decrease the first Betti number.  In particular,
no collection of direct torus surgeries on $X$ can produce a rational homology
$\CP^2$.

We also reconsider the construction proposed in Remark~1 on p.~46 of
Akhmedov--Sakall\i--Yeung \cite{ASY}; this is p.~18 of
arXiv:2102.05265v2.  That remark uses a degree-two unramified cover of the
Cartwright--Steger surface to propose a minimal symplectic four-manifold with
$\pi_1\cong\Z/2$ and $(\sigma,\chi_h)=(0,8)$.  The fourfold-cover
construction in \cite[Section~5.3, pp.~44--45]{ASY}, corresponding to
pp.~16--17 of arXiv:2102.05265v2, is different.  It produces a simply
connected symplectic manifold $M_{2,10}$ with
\[
(e,\sigma,\chi_h,c_1^2)=(38,2,10,82),
\]
which is an exotic smoothly irreducible copy of
\[
19\CP^2\#17\overline{\CP}^{\,2}.
\]
The intermediate double-cover construction with fundamental group $\Z/2$ and signature zero was only outlined in Remark~1, without the subgroup and surface-complement calculations supplied here.  The main point of the present paper is to complete those calculations.  The nodal curve supplies exactly the mod-two homology kernel, and a geometric dual created by blowing up makes the surface-sum quotient explicit.  It also carries out the construction on the original Cartwright--Steger surface, where it gives fundamental group $(\Z/2)^2$ and signature $-1$; this Klein four-group construction does not appear in \cite{ASY}.

The paper \cite{ASY} used complex surfaces on the Bogomolov--Miyaoka--Yau line, including the Cartwright--Steger surface and quotients of Hirzebruch line-arrangement surfaces, to construct symplectic four-manifolds with nonnegative signature.  Its Theorem~1 gives irreducible symplectic manifolds, together with infinitely many nonsymplectic manifolds, homeomorphic but not diffeomorphic to
\[
(2n-1)\CP^2\#(2n-1)\overline{\CP}^{\,2},\qquad n\geq 9,
\]
The families with positive signature are stated in \cite[Theorem~2]{ASY}.
The ball quotient obtained from a Hirzebruch line arrangement is constructed
in \cite[Section~4]{ASY}.  The first arXiv version was posted on February 10,
2021.

The construction considered here uses an intermediate double cover of $X$ and produces a minimal symplectic four-manifold with fundamental group $\Z/2$ and invariants
\[
(e,\sigma,\chi_h,c_1^2)=(32,0,8,64).
\]
An exact Todd--Coxeter computation identifies the nodal-curve subgroup with the mod-two kernel.  We smooth the positive node first and then blow up a regular point of the smoothed surface.  The exceptional sphere is a geometric dual, which gives the required meridian relation in the surface-sum calculation.
For every compatible symplectically admissible gluing, the universal double cover of the resulting manifold is an exotic symplectic
\[
31\CP^2\#31\overline{\CP}^{\,2}.
\]
The corresponding construction on $X$ itself produces a minimal symplectic manifold $Z_V$ with
\[
\pi_1(Z_V)\cong(\Z/2)^2,
\qquad
(e,\sigma,\chi_h,c_1^2)=(29,-1,7,55).
\]

Theorem~\ref{thm:atoroidal} shows that every continuous map $T^2\to X$ sends the integral fundamental class to zero.  Thus the Cartwright--Steger surface contains no homologically essential embedded torus.  Corollary~\ref{cor:noCP2} gives the resulting obstruction to torus surgery.

The organization of the proof is as follows.  Sections~2 and~3 establish the
torus-surgery obstruction.  Section~4 introduces the covering tower and the
symplectic sum.  Section~5 proves minimality and determines the $w_2$-type of
the universal cover.  Section~6 identifies the nodal-curve subgroup, analyzes
the positive complement, and computes the surface-sum quotient.  Section~7
applies the same argument on $X$ itself, and Section~8 records the exact
finite-index computations used in the subgroup calculation.

\section{Tori in negatively curved aspherical manifolds}

\begin{theorem}\label{thm:atoroidal}
Let $M$ be a closed aspherical Riemannian manifold with strictly negative sectional curvature.  For every continuous map $f:T^2\to M$,
\[
f_*[T^2]=0\in H_2(M;\Z).
\]
\end{theorem}

\begin{proof}
By Preissmann's theorem \cite{Preissmann}, every nontrivial abelian subgroup of $\pi_1(M)$ is infinite cyclic.  Hence the image of
\[
f_*:\pi_1(T^2)=\Z^2\longrightarrow\pi_1(M)
\]
is trivial or cyclic.  The induced homomorphism therefore factors through a homomorphism $\Z^2\to\Z$.  Since both $T^2$ and $M$ are Eilenberg--Mac Lane spaces, the map $f$ is homotopic to a composite
\[
T^2\longrightarrow S^1\longrightarrow M.
\]
The assertion follows from $H_2(S^1;\Z)=0$.
\end{proof}

\begin{corollary}\label{cor:CS-atoroidal}
Every smoothly embedded torus in the Cartwright--Steger surface is integrally nullhomologous.  In particular, the Cartwright--Steger surface contains no symplectically embedded torus.
\end{corollary}

\begin{proof}
Let $i:T^2\hookrightarrow X$ be the inclusion of a smoothly embedded torus.
Applying Theorem~\ref{thm:atoroidal} to $i$ gives
\[
i_*[T^2]=[T^2]=0\in H_2(X;\Z),
\]
which proves the first assertion.  Suppose now that $T^2$ is symplectic, and
orient it by the restriction of the symplectic form $\omega$.  Then
\[
\int_{T^2}\omega>0.
\]
On the other hand, the integral of a closed form depends only on the real
homology class of the surface.  Since the integral class $[T^2]$ vanishes, its
image in $H_2(X;\mathbb R)$ also vanishes, and hence
\[
\int_{T^2}\omega=\langle[\omega],[T^2]\rangle=0,
\]
a contradiction.
\end{proof}

\section{Homology under surgery on nullhomologous tori}

Let $M$ be a closed connected oriented four-manifold and let $T_1,\ldots,T_n\subset M$ be pairwise disjoint oriented embedded tori satisfying
\[
[T_i]=0\in H_2(M;\Z).
\]
Put
\[
C=M\setminus\Int\nu(T_1\sqcup\cdots\sqcup T_n).
\]

\begin{lemma}\label{lem:exterior}
There is an exact sequence
\[
0\longrightarrow\Z^n\longrightarrow H_1(C;\Z)
\longrightarrow H_1(M;\Z)\longrightarrow0,
\]
in which the standard generators of $\Z^n$ map to the meridians of the tori.
\end{lemma}

\begin{proof}
Excision and the Thom isomorphism give
\[
H_2(M,C;\Z)\cong\bigoplus_{i=1}^nH_0(T_i;\Z)\cong\Z^n,
\qquad H_1(M,C;\Z)=0.
\]
The map $H_2(M;\Z)\to H_2(M,C;\Z)$ is intersection with the classes $[T_i]$.  It vanishes because every $T_i$ is integrally nullhomologous.  The assertion now follows from the exact sequence of the pair $(M,C)$.
More precisely, the relevant part of that sequence is
\[
H_2(M;\Z)\longrightarrow \Z^n\xrightarrow{\partial}
H_1(C;\Z)\longrightarrow H_1(M;\Z)\longrightarrow0.
\]
Under the Thom identification, the boundary of the generator belonging to $T_i$ is the positively oriented meridian of $T_i$.  Since the first map is zero, $\partial$ is injective and gives the asserted copy of $\Z^n$.
\end{proof}

For each boundary component choose an orientation-reversing diffeomorphism
\[
\phi_i:\partial(T^2\times D^2)\longrightarrow\partial_iC
\]
and let $N$ be the resulting torus filling.  If $m_i$ denotes the meridian of the attached $T^2\times D^2$, write
\[
\gamma_i=(\phi_i)_*[m_i]\in H_1(\partial_iC;\Z).
\]

\begin{theorem}\label{thm:surgery}
For arbitrary torus fillings as above,
\[
H_1(N;\Z)\cong
H_1(C;\Z)/\langle\gamma_1,\ldots,\gamma_n\rangle.
\]
For every field $\mathbb F$,
\[
\dim_{\mathbb F}H_1(N;\mathbb F)
\geq \dim_{\mathbb F}H_1(M;\mathbb F).
\]
Moreover,
\[
e(N)=e(M),\qquad \sigma(N)=\sigma(M).
\]
\end{theorem}

\begin{proof}
The Mayer--Vietoris sequence for
\[
N=C\cup\coprod_{i=1}^n(T^2\times D^2)
\]
contains
\[
\bigoplus_i H_1(T^3;\Z)\longrightarrow
H_1(C;\Z)\oplus\bigoplus_iH_1(T^2\times D^2;\Z)
\longrightarrow H_1(N;\Z)\longrightarrow0.
\]
For the $i$th component, two generators of $H_1(T^3;\Z)$ map to a basis of
$H_1(T^2\times D^2;\Z)$ and may therefore be eliminated.  The third generator, the meridian $m_i$ of the attached solid torus, maps to zero in $H_1(T^2\times D^2;\Z)$ and to $\gamma_i$ in $H_1(C;\Z)$.  Thus the only remaining relation from that boundary component is $\gamma_i=0$, which proves the integral formula.

Lemma~\ref{lem:exterior}, with coefficients in $\mathbb F$, gives
\[
\dim_{\mathbb F}H_1(C;\mathbb F)
=\dim_{\mathbb F}H_1(M;\mathbb F)+n.
\]
The filling imposes at most $n$ independent relations, which proves the inequality.  Euler characteristic and signature are unchanged because the removed and attached pieces have Euler characteristic and signature zero; the signature statement also follows from Novikov additivity.
\end{proof}

\begin{corollary}\label{cor:noCP2}
Let $N$ be obtained by arbitrary torus surgeries on a fixed finite disjoint family of tori in the Cartwright--Steger surface.  Then, over every field $\mathbb F$,
\[
\dim_{\mathbb F}H_1(N;\mathbb F)\geq2.
\]
Consequently $N$ is neither simply connected nor a rational homology $\CP^2$.  In particular, direct Luttinger surgeries on such a family cannot produce a standard or exotic $\CP^2$.
\end{corollary}

\begin{proof}
Every torus is integrally nullhomologous by Corollary~\ref{cor:CS-atoroidal}, and $H_1(X;\Z)\cong\Z^2$.  Apply Theorem~\ref{thm:surgery}.
\end{proof}

\begin{remark}
Corollary~\ref{cor:noCP2} concerns fillings of the common complement of a fixed disjoint torus link in $X$.  It does not apply automatically to later surgeries on core tori created by previous fillings or to newly selected tori whose nullhomology has not been verified in the manifold present at that stage.
\end{remark}

If $d=b_1(N)-2$, then $d\geq0$ and the preservation of Euler characteristic and signature gives
\[
(b_2(N),b_2^+(N),b_2^-(N))=(5+2d,3+d,2+d).
\]

\section{The intermediate double-cover construction}

We first describe the covering tower that organizes the construction.  Let
$N$ be the kernel of reduction modulo two on the abelianization of $\Pi$.
The corresponding fourfold cover and any intermediate double cover fit into
\[
X_4\longrightarrow X_2\longrightarrow X,
\qquad
\pi_1(X_4)=N\subset\Pi_2=\pi_1(X_2)\subset\Pi=\pi_1(X),
\]
with
\[
\Pi/N\cong(\Z/2)^2,
\qquad [\Pi_2:N]=2.
\]
The decisive point is that the image of the fundamental group of the nodal
curve is precisely $N$.  It follows that the fourfold cover contains four
copies of the curve and each intermediate double cover contains two, with
every component mapping homeomorphically to the original curve.  After one
component is smoothed and blown up, its image in $\Pi_2$ is still $N$.  The
positive complement contributes no group after the meridian is killed, so
the symplectic sum has fundamental group $\Pi_2/N\cong\Z/2$.  The remainder
of this section defines the construction; the subgroup and complement
assertions are proved in Section~\ref{sec:subgroup}.

Let $D\subset X$ be the nodal curve used in the Cartwright--Steger
construction.  Its normalization and stabilizer are described in
\cite[Proposition~2.4 and Section~2.8]{CKY}; see also
\cite[Section~2.3, Proposition~1]{ASY}.  Its normalization has genus four and
$D^2=-1$.  Let
\[
\Pi=\pi_1(X),\qquad
A=\operatorname{im}\bigl(\pi_1(\widehat D)\to\Pi\bigr),
\qquad N=\ker\bigl(\Pi\to H_1(X;\Z/2)\bigr).
\]
The eight generators of the normalization group are displayed after \cite[Proposition~2.4]{CKY}; their images in homology are computed in \cite[Section~2.8]{CKY}.  Since $H_1(X;\Z)\cong\Z^2$, one has
\[
[\Pi:N]=4,\qquad \Pi/N\cong(\Z/2)^2.
\]
The assertion in \cite[Section~2.3, Proposition~1]{ASY} that the nodal-curve
subgroup equals this kernel will be verified below.
To identify the image of the nodal curve itself with $N$, one must also calculate the branch-joining generator and certify the resulting subgroup equality.  The normal-closure statement alone is insufficient.

Choose an order-two subgroup of $\Pi/N$, let $\Pi_2$ be its inverse image, and denote the corresponding double cover by $X_2\to X$.  In each of the three possible covers choose one lifted component $D_2$ of the nodal curve.  Its two local branches belong to the same component and meet positively.  Smooth this node symplectically, and denote the resulting embedded genus-five surface by $\overline D$.

We recall why this operation is symplectic and preserves the homology class.
The covering map is unramified and holomorphic, so $D_2$ is again a complex
curve with an ordinary node.  In holomorphic coordinates on a small ball $U$
about the node, its two branches have the standard local model
\[
\{z_1z_2=0\}=\{z_1=0\}\cup\{z_2=0\}\subset\mathbb C^2.
\]
The standard symplectic resolution of a positive node \cite{Gompf,McCarthyWolfson}
replaces them by
\[
\{z_1z_2=\varepsilon\},\qquad 0<|\varepsilon|\ll1,
\]
with the replacement adjusted through the local symplectic model to agree
with the original branches near $\partial U$.  The smoothed local surface is
a complex annulus in the standard model and is therefore symplectic.  The
original pair of disks and the smoothing determine relative two-cycles with
the same oriented boundary in $\partial U$.  Their relative classes agree
because
\[
H_2(U,\partial U;\Z)=H_2(B^4,S^3;\Z)=0.
\]
Thus the local replacement does not change the global class:
\[
[\overline D]=[D_2]\in H_2(X_2;\Z).
\]
The annulus joins the two punctures of the genus-four normalization and adds
one handle.  Hence
\[
g(\overline D)=5,
\qquad [\overline D]^2=[D_2]^2=-1.
\]
Then
\[
[\Pi_2:N]=2.
\]
Choose a regular point $q\in\overline D$ away from the smoothing chart and blow it up.  The proper transform is a genus-five symplectic surface
\[
\Sigma_-\subset X_2\#\overline{\CP}^{\,2},
\qquad [\Sigma_-]^2=-2.
\]
If $E$ denotes the exceptional class of the blowup, then
\[
[\Sigma_-]=p^*[\overline D]-E,
\qquad
[\Sigma_-]^2=[\overline D]^2+E^2=-1-1=-2.
\]
Because the blowup point is a regular point of $\overline D$, the exceptional
sphere meets the proper transform once.  This intersection will give a disk
in the complement bounded by the meridian of $\Sigma_-$.
Unlike the construction in \cite[Lemma~3]{ASY}, we smooth the node before blowing up a regular point.  This preserves the genus and final square while making the exceptional sphere a geometric dual.

\begin{lemma}\label{lem:positive-input}
The simply connected minimal symplectic manifold $X_{2,3}$ of
\cite[Theorem~6]{ASY} has
\[
e(X_{2,3})=9,\qquad \sigma(X_{2,3})=-1.
\]
The construction in \cite[Section~5.3, pp.~44--45]{ASY} contains square-zero
symplectic surfaces $F_2$ and $F_3$ of genera two and three that meet
transversely and positively in one point.  Their
symplectic resolution $F$ has genus five and square $+2$.  A parallel copy
$P$ of $F_2$ may be chosen so that
\[
P^2=0,\qquad P\cdot F=1.
\]
\end{lemma}

\begin{proof}
The simple connectivity, minimality, numerical invariants, and the two
intersecting square-zero surfaces are the properties established in the
cited construction.  Resolving their single positive intersection gives
$[F]=[F_2]+[F_3]$.  Hence
\[
g(F)=g(F_2)+g(F_3)=5,
\qquad
F^2=F_2^2+2F_2\cdot F_3+F_3^2=2.
\]
Choose $P$ as a sufficiently small symplectic normal translate of $F_2$.
To see the intersection with $F$, take positive complex coordinates near
$F_2\cap F_3$ in which $F_2=\{z=0\}$, $F_3=\{w=0\}$, and the resolution is
$zw=\varepsilon$.  Locally the parallel surface is $\{z=\delta\}$, and it
meets the resolution at the single point $w=\varepsilon/\delta$.  Taking
$0<|\varepsilon|\ll|\delta|$ places this point inside the smoothing chart.
Outside that chart $P$ is disjoint from the $F_2$ portion and has no further
intersection with the $F_3$ portion.  Hence $P^2=0$ and $P\cdot F=1$.
\end{proof}

Form the symplectic sum
\[
Z=(X_2\#\overline{\CP}^{\,2})
\mathbin{\#_{\Sigma_-=F}}X_{2,3}.
\]
For the symplectic statements, the areas are scaled to agree and the gluing is
chosen to be symplectically admissible in Gompf's sense.  The fundamental-group
calculation below applies more generally to every orientation-reversing
isomorphism of the normal circle bundles.  Thus the topological gluing statement
is kept separate from the assertion that the sum is symplectic.

\begin{proposition}\label{prop:numerics}
The numerical invariants of $Z$ are
\[
e(Z)=32,\qquad \sigma(Z)=0,
\qquad \chi_h(Z)=8,\qquad c_1^2(Z)=64.
\]
\end{proposition}

\begin{proof}
The double cover has $e(X_2)=6$ and $\sigma(X_2)=2$.  Hence
\[
e(X_2\#\overline{\CP}^{\,2})=7,\qquad
\sigma(X_2\#\overline{\CP}^{\,2})=1.
\]
By Lemma~\ref{lem:positive-input}, the second block satisfies
$e(X_{2,3})=9$ and $\sigma(X_{2,3})=-1$.  Since $e(F)=-8$, the fiber-sum
formulas give
\[
e(Z)=7+9-2(-8)=32,\qquad \sigma(Z)=1-1=0.
\]
The remaining formulas follow from
\[
\chi_h=\frac{e+\sigma}{4},\qquad c_1^2=2e+3\sigma.
\]
\end{proof}

\begin{theorem}\label{thm:group}
For every intermediate double cover and every orientation-reversing
normal-circle-bundle gluing of the two complements, the underlying smooth sum
satisfies
\[
\pi_1(Z)\cong\Pi_2/N\cong\Z/2.
\]
\end{theorem}

\begin{proof}
Let $A^\circ$ and $B^\circ$ denote the negative and positive surface
complements.  Lemma~\ref{lem:dual-sphere} shows that the meridian is trivial in
$\pi_1(A^\circ)$, and hence
\[
\pi_1(A^\circ)\cong\Pi_2.
\]
The gluing identifies the positive meridian with the inverse negative
meridian.  Proposition~\ref{prop:F-complement} shows that
$\pi_1(B^\circ)$ is normally generated by its meridian, so its entire image
dies in the sum.  The remaining relations are therefore the normal closure
of the image of the summing surface on the negative side.  By
Lemma~\ref{lem:smoothing-image} this image is the nodal-curve subgroup $B$,
and Proposition~\ref{prop:subgroup-index} identifies $B$ with $N$.
Proposition~\ref{prop:exact-quotient} then gives
\[
\pi_1(Z)\cong
\Pi_2/\langle\!\langle N\rangle\!\rangle_{\Pi_2}
=\Pi_2/N\cong\Z/2.
\]
The same argument also shows why meridional twists in the gluing do not alter
the quotient: their additional factors are already trivial.
\end{proof}

\begin{remark}
The reproducible Todd--Coxeter computation accompanying the paper verifies
that the nodal-curve subgroup is $N$.  Lemma~\ref{lem:smoothing-image}
shows that smoothing introduces no further condition, while
Lemma~\ref{lem:dual-sphere} kills the negative meridian.
\end{remark}

It follows that rationally
\[
b_1(Z)=0,\qquad b_2(Z)=30,\qquad b_2^+(Z)=b_2^-(Z)=15.
\]
Thus $Z$ has the rational homology of
\[
15\CP^2\#15\overline{\CP}^{\,2}.
\]
The integral homology and, for compatible gluings, the $w_2$-type are determined below.  A homeomorphism classification of the quotient requires the equivariant intersection data.

\section{Minimality and the \texorpdfstring{$w_2$}{w2}-type}

The subsequent $w_2$-type statement uses a gluing compatible with the indicated dual surfaces.

\begin{theorem}\label{thm:minimal}
For every symplectically admissible gluing, the symplectic manifold $Z$ is
minimal.
\end{theorem}

\begin{proof}
Let $p:X_2\#\overline{\CP}^{\,2}\to X_2$ be the blowdown map.  The compact ball quotient $X_2$ is aspherical.  Hence, for every embedded sphere $S$, the map $p|_S$ is nullhomotopic and $p_*[S]=0$.  Since $\ker p_*=\Z E$, the equality $S^2=-1$ implies $[S]=\pm E$.  On the other hand,
\[
E\cdot[\Sigma_-]=E\cdot(p^*[\overline D]-E)=1.
\]
Thus no exceptional sphere is disjoint from $\Sigma_-$.  The other block
$X_{2,3}$ is minimal by Lemma~\ref{lem:positive-input}.  The exceptional
case in Usher's theorem requires one summand to be an $S^2$-bundle over a
surface with the summing surface as a section.  Such a ruled manifold has
signature zero, whereas the two summands here have signatures $1$ and $-1$.
The exceptional case therefore does not occur.  Usher's minimality theorem
for positive-genus symplectic sums now proves that $Z$ is minimal
\cite[Theorem~1.1]{Usher}.
\end{proof}

Let $P\subset X_{2,3}$ be the surface supplied by
Lemma~\ref{lem:positive-input}.

\medskip
\noindent\textbf{Definition.}
A gluing is \emph{compatible with the dual surfaces} $E$ and $P$ if its base
diffeomorphism sends the marked point $E\cap\Sigma_-$ to $P\cap F$ and, in
oriented product coordinates on the normal bundles over small disks about
these points, its fiber-reversing lift has constant phase and identifies the
boundary circles of the punctured surfaces.

Compatible symplectically admissible gluings exist.  First scale the
symplectic forms so that the two summing surfaces have equal area, and choose
an area-preserving base diffeomorphism carrying the marked point
$E\cap\Sigma_-$ to $P\cap F$.  The normal bundles have Euler numbers $-2$
and $+2$, so after reversing the circle fibers the base map admits a bundle
isomorphism.  Over a disk about the marked point this isomorphism is described
by a phase map to $S^1$.  A unitary gauge transformation supported in a
slightly larger disk makes that phase constant on the smaller disk, without
changing the normal Euler class.  Because the adjustment is unitary, it
preserves the fiberwise symplectic neighborhood data.  The symplectic
neighborhood theorem then gives an admissible gluing in the sense of
\cite{Gompf}.

\begin{lemma}\label{lem:framing-sum}
Let $Q_i$ be oriented surfaces meeting the summing surfaces $\Sigma_i$ once,
and remove small disks about the intersection points.  Suppose that the
normal-circle-bundle gluing matches the two boundary circles and, in product
trivializations over these circles, reverses the normal fiber with constant
phase.  Then the punctured surfaces glue to an oriented surface $Q$ and
\[
Q^2=Q_1^2+Q_2^2.
\]
\end{lemma}

\begin{proof}
The self-intersection of a closed oriented surface is the Euler number of its
oriented normal bundle.  Choose oriented product coordinates at
$Q_i\cap\Sigma_i$ and use their product framing on the boundary of the disk
removed from $Q_i$.  This framing extends across the deleted normal disk.
Consequently, removing that disk contributes no local relative Euler number,
and the relative Euler number of the punctured normal bundle remains
$Q_i^2$.  Under the gluing, a possible
correction to their sum is the degree of the transition map between these two
boundary sections.  The constant-phase hypothesis makes this transition map
constant, hence of degree zero.  Relative Euler numbers therefore add and give
$Q^2=Q_1^2+Q_2^2$.
\end{proof}

\begin{proposition}\label{prop:w2}
Choose a gluing compatible with the indicated dual surfaces.  Then $Z$ has $w_2$-type I, meaning that its universal cover is non-spin.  Its universal double cover is homeomorphic to
\[
31\CP^2\#31\overline{\CP}^{\,2}.
\]
\end{proposition}

\begin{proof}
Remove small disks from the dual sphere $E$ and from $P$ at their intersections with the two summing surfaces.  The two punctured surfaces join to an embedded genus-two surface $S\subset Z$.  Lemma~\ref{lem:framing-sum} gives
\[
S^2=E^2+P^2=-1.
\]
The surface $S$ is obtained from $P\setminus\Int D^2$ by attaching a disk, so $\pi_1(P\setminus\Int D^2)\to\pi_1(S)$ is surjective.  Its map to $Z$ factors through the positive complement, whose image is trivial in $\pi_1(Z)$ by Proposition~\ref{prop:exact-quotient}.  Thus $\pi_1(S)\to\pi_1(Z)$ is trivial.  The inverse image of $S$ in the universal double cover consists of two disjoint components, each mapped diffeomorphically to $S$ and therefore having square $-1$.  Consequently the universal cover is non-spin.

The universal cover is simply connected and, by multiplicativity of Euler characteristic and signature under finite covers, has
\[
e=64,\qquad \sigma=0,\qquad b_2^+=b_2^-=31.
\]
Here $b_1=b_3=0$, so $b_2=62$, and the signature determines the two indices.  The lifted square $-1$ surface shows that the intersection form is odd.  It is also indefinite and unimodular, and therefore is isomorphic to
\[
31\langle1\rangle\oplus31\langle-1\rangle.
\]
The Kirby--Siebenmann invariant vanishes because the cover is smooth.  Freedman's classification theorem now gives the asserted homeomorphism \cite{Freedman}.
\end{proof}

\begin{corollary}\label{cor:exotic-cover}
For every compatible symplectically admissible gluing, the universal double cover $\widetilde Z$ is an exotic symplectic
\[
31\CP^2\#31\overline{\CP}^{\,2}.
\]
\end{corollary}

\begin{proof}
The covering map $\widetilde Z\to Z$ is finite and unbranched, so the symplectic form on $Z$ pulls back to a symplectic form on $\widetilde Z$.  Since $b_2^+(\widetilde Z)=31$, Taubes's theorem gives a nonzero Seiberg--Witten invariant \cite{Taubes}.  In contrast, the Seiberg--Witten invariants of the standard connected sum
\[
31\CP^2\#31\overline{\CP}^{\,2}
\]
vanish by the connected-sum theorem, since it splits as a connected sum of two manifolds with positive $b_2^+$ \cite{Witten}.  Proposition~\ref{prop:w2} gives the homeomorphism, so the two smooth structures are not diffeomorphic.
\end{proof}

\begin{proposition}\label{prop:rational-equivariant}
Let $\tau$ be the deck involution of the universal cover.  Over $\Q$,
\[
H_2(\widetilde Z;\Q)
\cong \Q_+^{30}\oplus\Q_-^{32},
\]
where the subscripts denote the trivial and sign representations.  The restrictions of the intersection form have signatures $(15,15)$ and $(16,16)$, respectively.
\end{proposition}

\begin{proof}
Let $p:\widetilde Z\to Z$ be the covering map.  Over $\Q$, transfer gives
\[
p_*p^*=2\operatorname{id},\qquad
p^*p_*=\operatorname{id}+\tau_*.
\]
It follows that $p^*$ identifies $H_2(Z;\Q)$ with the $+1$ eigenspace of $\tau_*$, which has dimension thirty.  Since $b_2(\widetilde Z)=62$, the $-1$ eigenspace has dimension thirty-two.  Invariance of the intersection form gives
\[
Q(\tau_*x,\tau_*y)=Q(x,y).
\]
Thus the two eigenspaces are orthogonal: if $\tau_*x=x$ and $\tau_*y=-y$, then $Q(x,y)=-Q(x,y)$.  Moreover,
\[
Q_{\widetilde Z}(p^*u,p^*v)=2Q_Z(u,v),
\]
so the invariant form has signature zero.  The total signature is zero, and orthogonal additivity gives signature zero on the anti-invariant form.  The stated positive and negative indices follow from the dimensions of the two eigenspaces.
\end{proof}

Under the hypotheses of Proposition~\ref{prop:w2}, the $w_2$-type, the homeomorphism type of the universal cover, and the rational equivariant intersection form are fixed.  A homeomorphism classification of the quotient $Z$ itself would additionally require the integral equivariant intersection form over $\Z[\Z/2]$; its integral extension data are not determined by the rational representation and the ordinary intersection form alone.

\section{The fundamental group of the symplectic sum}\label{sec:subgroup}

The computation of the fundamental group has three parts.  We first identify
the subgroup carried by the nodal curve and determine how the curve lifts to
the relevant covers.  We then study the complement of the positive summing
surface.  Finally, we combine these two calculations by van Kampen's theorem.

\subsection{The nodal-curve subgroup and covering analysis}

Let
\[
A=\operatorname{im}\bigl(\pi_1(\widehat D)\to\Pi\bigr),
\qquad
B=\operatorname{im}\bigl(\pi_1(D)\to\Pi\bigr),
\qquad
N=\ker\bigl(\Pi\to H_1(X;\Z/2)\bigr).
\]
We use the based paths of \cite[Section~2.3]{ASY}.  The two paths from the
base point to the preimages of the node have ambient words
\[
r_1=a_3^3a_1a_2^{-1},
\qquad
r_2=a_2a_1^{-2}a_3^{-1}a_1a_3^{-1}a_1^{-1}a_2^{-2}.
\]
A loop which follows the first path to the node and returns along the second
has word
\[
h=r_1r_2^{-1}.
\]
Changing the common path from the ambient base point conjugates all the words
simultaneously and does not change the subgroup up to conjugacy.

\begin{lemma}\label{lem:based-node}
With these based paths, the image of the nodal curve is
\[
B=\langle p_1,\ldots,p_8,h\rangle.
\]
In particular, the subgroup entered in the finite-index computation is the
geometric subgroup $\operatorname{im}(\pi_1(D)\to\Pi)$.
\end{lemma}

\begin{proof}
Let $x_1,x_2\in\widehat D$ be the two preimages of the node and choose a base
point $x_0\in\widehat D\setminus\{x_1,x_2\}$.  Choose paths $\gamma_i$ from
$x_0$ to $x_i$ whose images, after adjoining the fixed ambient base path, are
the words $r_i$ above.  These are precisely the two elements denoted
$\pi_1,\pi_2$ in \cite[Section~2.3]{ASY}.  Collapsing a maximal tree containing
$\gamma_1\cup\gamma_2$ shows that identifying $x_1$ with $x_2$ adds one free
generator to $\pi_1(\widehat D)$.  Its image in $\Pi$ is the loop
$\gamma_1\gamma_2^{-1}$ and hence is $r_1r_2^{-1}=h$.

The elements $p_1,\ldots,p_8$ displayed after
\cite[Proposition~2.4]{CKY} generate the stabilizer of the mirror
$b(M_c)$, which is the fundamental group of the normalization.  Therefore
their images generate $A$, and van Kampen gives
$B=\langle A,h\rangle=\langle p_1,\ldots,p_8,h\rangle$.
If the paths $\gamma_i$ are changed while their endpoints are fixed, $h$ is
replaced by $a_1ha_2$ with $a_1,a_2\in A$; this does not change
$\langle A,h\rangle$.  Changing the ambient base path conjugates the entire
subgroup.  Thus the asserted subgroup is independent of the allowed choices,
up to conjugacy.
\end{proof}

\begin{proposition}\label{prop:ABN}
With compatible base paths,
\[
B=\langle A,h\rangle,
\qquad
A\subset B\subset N,
\qquad
\Pi/N\cong(\Z/2)^2.
\]
\end{proposition}

\begin{proof}
Lemma~\ref{lem:based-node} gives $B=\langle A,h\rangle$.  Since
$H_1(X;\Z)\cong\Z^2$, reduction modulo two gives
$\Pi/N\cong(\Z/2)^2$, as noted in Section~4.  The CKY homology table in
\cite[Section~2.8]{CKY} gives
\[
\operatorname{im}\bigl(H_1(\widehat D;\Z)\to H_1(X;\Z)\bigr)=2\Z^2.
\]
Using
\[
[a_1]=(1,3),\qquad [a_2]=(-2,1),\qquad [a_3]=(-1,-1),
\]
one obtains
\[
[r_1]=(0,-1),\qquad [r_2]=(2,-5),\qquad [h]=(-2,4).
\]
The homology table shows that every generator of $A$ lies in $2\Z^2$ and
hence in $N$.  The displayed calculation also gives $h\in N$.  Therefore
$B=\langle A,h\rangle\subset N$.
\end{proof}

The conclusion $B=N$ does not follow formally from Proposition~\ref{prop:ABN}.  It is supplied by the following exact computation.

\begin{proposition}\label{prop:subgroup-index}
The nodal-curve subgroup satisfies
\[
B=N,
\qquad [\Pi:B]=4.
\]
\end{proposition}

\begin{proof}
Starting from the CKY presentation of the ambient lattice $\Gamma$, the GAP computation gives
\[
[\Gamma:\Pi]=864.
\]
The CKY words $p_1,\ldots,p_8$ and the branch-joining word $h$ are verified to lie in $\Pi$.  By Lemma~\ref{lem:based-node}, they generate the geometric subgroup $B$.  Two Todd--Coxeter runs, using the coset-table-based and relator-based strategies of the same SymPy implementation, give
\[
[\Gamma:B]=3456.
\]
The previously certified membership $B\subset\Pi$ therefore gives
\[
[\Pi:B]=\frac{3456}{864}=4.
\]
Proposition~\ref{prop:ABN} gives $B\subset N$, while $[\Pi:N]=4$.  Hence $B=N$.  The scripts and recorded output are supplied with the paper.
\end{proof}

\begin{proposition}\label{prop:lift-count}
The fourfold cover associated with $N$ contains four disjoint copies of the nodal curve $D$.  Each intermediate double cover contains two disjoint copies.  Every component maps homeomorphically to $D$, has geometric genus four, one positive node, and self-intersection $-1$.
\end{proposition}

\begin{proof}
For a normal cover corresponding to $G\triangleleft\Pi$, the components of the inverse image of $D$ are indexed by the double cosets
\[
G\backslash\Pi/B.
\]
The degree of the component represented by $g\in\Pi$ is
\[
[B:B\cap g^{-1}Gg].
\]
Proposition~\ref{prop:ABN} gives $B\subset N\subset G$ for $G=N$ and for every intermediate group $G=\Pi_2$.  These groups are normal because they are inverse images of subgroups of the abelian quotient $\Pi/N$.  Thus $B\subset g^{-1}Gg=G$, and the displayed degree is one.  Every component therefore maps homeomorphically to $D$.  The double-coset set reduces to $G\backslash\Pi$, so the number of components is $[\Pi:G]$, namely four for $G=N$ and two for $G=\Pi_2$.  Because the cover is unramified and locally biholomorphic, it preserves the node, its positive local intersection sign, and the normal Euler number.  Hence every component has one positive node, geometric genus four, and square $-1$.
\end{proof}

Thus the component count asserted in the earlier construction is correct.  By Proposition~\ref{prop:subgroup-index}, the image of the fundamental group of every lifted component is a conjugate of $N$.

\subsection{The positive surface complement}

The manifold $X_{2,3}$ contains square-zero symplectic surfaces of genera two and three meeting transversely once.  Resolving their intersection gives a symplectic surface
\[
F\subset X_{2,3},\qquad g(F)=5,\qquad F^2=2.
\]
Put $C=X_{2,3}\setminus\Int\nu F$ and let $\mu$ denote a meridian.

\begin{proposition}\label{prop:F-complement}
The complement satisfies
\[
H_1(C;\Z)=0.
\]
Its fundamental group is perfect and normally generated by $\mu$.
\end{proposition}

\begin{proof}
The pair sequence and the Thom isomorphism give
\[
H_2(X_{2,3};\Z)\xrightarrow{x\mapsto x\cdot[F]}
\Z\longrightarrow H_1(C;\Z)\longrightarrow0.
\]
The class $[F]$ is primitive.  Indeed, if $[F]=kA$ for $k>1$, then
\[
2=F^2=k^2A^2,
\]
which is impossible in an integral lattice.  Since $X_{2,3}$ is simply connected, its intersection lattice is unimodular, and the intersection map is surjective.  Hence $H_1(C;\Z)=0$.

Restoring the tubular neighborhood kills the normal closure of the meridian, so van Kampen gives
\[
\pi_1(X_{2,3})
\cong\pi_1(C)/\langle\!\langle\mu\rangle\!\rangle.
\]
The closed manifold is simply connected.  Therefore $\pi_1(C)$ is normally generated by $\mu$, and it is perfect because its abelianization vanishes.
\end{proof}

The proposition does not prove that $C$ is simply connected.  Nontrivial perfect groups of weight one exist.  Absent a geometric argument killing the meridian after gluing, an explicit relative presentation would be needed.

\subsection{The exact surface-sum quotient}

Let
\[
A^\circ=(X_2\#\overline{\CP}^{\,2})\setminus\Int\nu\Sigma_-,
\qquad B^\circ=X_{2,3}\setminus\Int\nu F,
\]
where $\Sigma_-^2=-2$.  The boundary is the circle bundle over $\Sigma_-$ with Euler number two.  With a fixed orientation convention its group has presentation
\[
\left\langle \alpha_i,\beta_i,\mu\ \middle|\
[\mu,\alpha_i]=[\mu,\beta_i]=1,
\prod_{i=1}^5[\alpha_i,\beta_i]=\mu^2
\right\rangle.
\]
A fiber-reversing gluing sends one meridian to the inverse of the other.  The surface lifts may also acquire meridional powers determined by a class in $H^1(\Sigma_-;\Z)$.

Let $G_A=\pi_1(A^\circ)$, $G_B=\pi_1(B^\circ)$, and let $P$ denote the fundamental group of the sum.  For a fixed gluing $\Phi$,
\[
P=(G_A*G_B)/
\langle\!\langle i_A(y)=i_B(\Phi_*y),\ y\in\pi_1(\partial A^\circ)\rangle\!\rangle.
\]
Let $H\subset\Pi_2$ be the image of the smoothed negative-square surface after its neighborhood is restored and put
\[
K=\langle\!\langle H\rangle\!\rangle_{\Pi_2}.
\]

There are three steps in the quotient calculation.  First, the exceptional
sphere on the negative side kills the meridian and identifies $G_A$ with
$\Pi_2$.  Second, the positive group is normally generated by its meridian,
so the gluing makes its entire image trivial.  Third, the surface loops from
the negative boundary impose the normal closure of $H$.  Thus van Kampen
reduces the surface sum to the single quotient
\[
\Pi_2\longrightarrow \Pi_2/\langle\!\langle H\rangle\!\rangle.
\]
The remaining task is to identify $H$ with the subgroup carried by the nodal
curve.

The exceptional sphere kills the meridian on the negative side.

\begin{lemma}\label{lem:dual-sphere}
The exceptional sphere $E$ meets $\Sigma_-$ transversely in one positive point.  Consequently the meridian of $\Sigma_-$ is nullhomotopic in $A^\circ$.
\end{lemma}

\begin{proof}
The surface $\Sigma_-$ is defined by first smoothing the positive node and then blowing up a regular point $q$ of the smoothed surface.  Choose an almost-complex structure compatible with the local symplectic form and making $\overline D$ holomorphic near $q$.  In standard complex coordinates $(u,v)$ centered at $q$, write $\overline D=\{v=0\}$.  In the blowup chart
\[
(u,v)=(s,st),
\]
the exceptional sphere is $\{s=0\}$ and the proper transform $\Sigma_-$ is $\{t=0\}$.  They therefore meet once, transversely and positively.  The punctured exceptional sphere $E\setminus\Int\nu\Sigma_-$ is a properly embedded disk in $A^\circ$ whose boundary is the oriented circle fiber.  Hence the meridian is nullhomotopic.
\end{proof}

For any connected codimension-two submanifold, the kernel of the homomorphism from the complement group to the ambient group is normally generated by a meridian.  Lemma~\ref{lem:dual-sphere} therefore gives
\[
\pi_1(A^\circ)\cong\Pi_2.
\]
On the positive side, $X_{2,3}$ is simply connected, so $G_B$ is normally generated by its meridian.  The normal-bundle gluing identifies that meridian with the inverse meridian on the negative side.  Hence the entire image of $G_B$ dies in the sum.

\begin{proposition}\label{prop:exact-quotient}
For every fiber-reversing normal-bundle gluing,
\[
P\cong\Pi_2/K.
\]
\end{proposition}

\begin{proof}
Write $\mu_A$ and $\mu_B$ for the two meridians.  Lemma~\ref{lem:dual-sphere} gives $i_A(\mu_A)=1$ in $G_A$.  A fiber-reversing gluing has
\[
\Phi_*(\mu_A)=\mu_B^{-1},
\]
up to the fixed orientation convention.  The corresponding amalgamation relation therefore gives $i_B(\mu_B)=1$ in $P$.  By Proposition~\ref{prop:F-complement}, $G_B$ is normally generated by $\mu_B$, so the homomorphism $G_B\to P$ is trivial.

After $G_B$ has died, every boundary loop parallel to the surface on the positive side becomes trivial.  On the negative side these loops generate the image of the fundamental group $H$ of the restored surface.  A general normal-bundle gluing may replace a chosen surface lift by that lift times a power of $\mu_B$.  These additional powers have already become trivial, so the conclusion is independent of the class in $H^1(\Sigma_-;\Z)$ specifying the gluing twist.  Hence the only remaining relations in $G_A\cong\Pi_2$ are the normal closure of $H$, and
\[
P\cong \Pi_2/\langle\!\langle H\rangle\!\rangle_{\Pi_2}=\Pi_2/K.
\]
\end{proof}

\begin{lemma}\label{lem:smoothing-image}
Let $B=\operatorname{im}(\pi_1(D)\to\Pi)$ be the nodal-curve subgroup defined in Section~\ref{sec:subgroup}.  Up to conjugacy, the image $H$ of $\pi_1(\Sigma_-)$ in $\Pi_2$ is $B$.  Consequently
\[
K=\langle\!\langle B\rangle\!\rangle_{\Pi_2}.
\]
\end{lemma}

\begin{proof}
Choose a small ball $U$ about the node in which the two branches are the coordinate disks $\{uv=0\}$.  Replacing them by $\{uv=\varepsilon\}$ attaches a one-handle to the normalization with small disks removed at the two preimages of the node.  The old genus-four part supplies the subgroup $A$.

Choose an arc in the punctured normalization joining the two attaching circles.  Closing this arc through the core of the smoothing annulus gives one generator of the new handle.  Its image in the ambient manifold follows one branch to the node and returns along the other; with the base paths used to define $r_1$ and $r_2$, this is $r_1r_2^{-1}=h$.  A dual generator may be chosen as the waist circle of the smoothing annulus.  It bounds a disk in the coordinate ball $U$, and hence maps trivially to the ambient fundamental group.  Van Kampen applied to the smoothed surface therefore gives
\[
\operatorname{im}\bigl(\pi_1(\overline D)\to\Pi_2\bigr)
=\langle A,h\rangle=B,
\]
up to conjugacy caused by the choice of base path.  Blowing up the regular point $q$ does not change the ambient fundamental group, and the proper transform is diffeomorphic to $\overline D$.  Its image is therefore still $B$.
\end{proof}

Proposition~\ref{prop:subgroup-index} gives $B=N$, and Lemma~\ref{lem:smoothing-image} therefore gives $K=N$.  It follows from Proposition~\ref{prop:exact-quotient} that
\[
P\cong\Pi_2/N\cong\Z/2.
\]
This proves Theorem~\ref{thm:group}.  No separate computation of the full group of $X_{2,3}\setminus\nu F$ is necessary.

Abelianization gives $H_1(Z;\Z)=\Z/2$.  Since $Z$ is oriented, Poincar\'e duality and the universal coefficient theorem give
\[
H_3(Z;\Z)\cong H^1(Z;\Z)=0.
\]
The free rank of $H_2(Z;\Z)$ is therefore determined by
\[
32=e(Z)=1-0+b_2-0+1,
\]
so it equals thirty.  Poincar\'e duality identifies $H_2(Z;\Z)$ with
$H^2(Z;\Z)$.  The universal coefficient sequence
\[
0\longrightarrow \operatorname{Ext}(H_1(Z;\Z),\Z)
\longrightarrow H^2(Z;\Z)
\longrightarrow \operatorname{Hom}(H_2(Z;\Z),\Z)
\longrightarrow0
\]
therefore identifies the torsion subgroup of $H_2(Z;\Z)$ with
$\operatorname{Ext}(\Z/2,\Z)\cong\Z/2$.  Consequently
\[
H_2(Z;\Z)=\Z^{30}\oplus\Z/2.
\]
Moreover, $b_2^+(Z)=b_2^-(Z)=15$.  Its universal double cover has
\[
e=64,\qquad \sigma=0,\qquad b_2^+=b_2^-=31.
\]
For a compatible gluing, Proposition~\ref{prop:w2} determines the parity and the $w_2$-type.  The equivariant intersection form of the quotient is not computed here, so no homeomorphism classification of $Z$ itself is asserted.

\section{A signature \texorpdfstring{$-1$}{-1} example with Klein four fundamental group}

The construction has three natural levels.  The fourfold cover used in
\cite[Section~5.3, pp.~44--45]{ASY} gives the simply connected exotic
manifold $M_{2,10}$, homeomorphic to
$19\CP^2\#17\overline{\CP}^{\,2}$ and having signature two.  The
intermediate double cover gives the signature-zero manifold with fundamental
group $\Z/2$ studied above.  We now apply the same argument to the original
Cartwright--Steger surface.

Smooth the positive node of $D\subset X$ and denote the resulting embedded genus-five symplectic surface by $\overline D_V$.  It represents the same homology class as $D$, and hence
\[
[\overline D_V]^2=-1.
\]
Blow up a regular point of $\overline D_V$.  Its proper transform is a symplectic surface
\[
\Sigma_V\subset X\#\overline{\CP}^{\,2},
\qquad g(\Sigma_V)=5,\qquad \Sigma_V^2=-2.
\]
As before, the exceptional sphere is a geometric dual to $\Sigma_V$.  Using the square $+2$ surface $F\subset X_{2,3}$, define
\[
Z_V=(X\#\overline{\CP}^{\,2})
\mathbin{\#_{\Sigma_V=F}}X_{2,3}.
\]

\begin{theorem}\label{thm:klein-four}
For every symplectically admissible fiber-reversing normal-bundle gluing,
$Z_V$ is a minimal symplectic four-manifold satisfying
\[
\pi_1(Z_V)\cong(\Z/2)^2
\]
and
\[
(e(Z_V),\sigma(Z_V),\chi_h(Z_V),c_1^2(Z_V))=(29,-1,7,55).
\]
Moreover,
\[
H_1(Z_V;\Z)\cong(\Z/2)^2,
\qquad
H_2(Z_V;\Z)\cong\Z^{27}\oplus(\Z/2)^2.
\]
\end{theorem}

\begin{proof}
Let
\[
A_V^\circ=(X\#\overline{\CP}^{\,2})\setminus\Int\nu\Sigma_V.
\]
The exceptional sphere meets $\Sigma_V$ once.  Its punctured complement is therefore a disk in $A_V^\circ$ bounded by the meridian of $\Sigma_V$.  The meridian is nullhomotopic, and the standard codimension-two complement argument gives
\[
\pi_1(A_V^\circ)\cong\pi_1(X)=\Pi.
\]
On the positive side, Proposition~\ref{prop:F-complement} shows that the group of $X_{2,3}\setminus\Int\nu F$ is normally generated by its meridian.  The gluing identifies this meridian with the inverse of the trivial meridian on the negative side, so the image of the entire positive complement group is trivial in the sum.

The local smoothing argument in Lemma~\ref{lem:smoothing-image} applies without passing to a cover.  The image of $\pi_1(\Sigma_V)$ in $\Pi$ is the nodal-curve subgroup $B$.  Consequently van Kampen's theorem gives
\[
\pi_1(Z_V)
\cong\Pi/\langle\!\langle B\rangle\!\rangle_\Pi
=\Pi/N
\cong(\Z/2)^2,
\]
where Proposition~\ref{prop:subgroup-index} gives $B=N$.  As in Proposition~\ref{prop:exact-quotient}, meridional powers introduced by a normal-bundle twist are already trivial, so the quotient is independent of the gluing class.

The negative block has
\[
e(X\#\overline{\CP}^{\,2})=4,
\qquad
\sigma(X\#\overline{\CP}^{\,2})=0.
\]
Using $e(X_{2,3})=9$, $\sigma(X_{2,3})=-1$, and $e(\Sigma_5)=-8$, the symplectic-sum formulas give
\[
e(Z_V)=4+9-2(-8)=29,
\qquad
\sigma(Z_V)=0-1=-1.
\]
Therefore
\[
\chi_h(Z_V)=\frac{e(Z_V)+\sigma(Z_V)}4=7,
\qquad
c_1^2(Z_V)=2e(Z_V)+3\sigma(Z_V)=55.
\]

The minimality argument is the same as in Theorem~\ref{thm:minimal}.  If an exceptional sphere in $X\#\overline{\CP}^{\,2}$ is blown down to $X$, its image is nullhomotopic because $X$ is aspherical.  Its class is therefore $\pm E$, and it meets $\Sigma_V$ once.  Thus the negative block contains no exceptional sphere disjoint from the summing surface.  The positive block is minimal by \cite[Theorem~6]{ASY}.  The ruled exception in Usher's theorem does not occur.  Indeed, the fundamental group of the negative block is the fundamental group of a closed aspherical four-manifold and has cohomological dimension four, whereas the fundamental group of a ruled surface has cohomological dimension at most two.  The positive block has signature $-1$, while an oriented $S^2$-bundle over a surface has signature zero.  Usher's theorem now implies that $Z_V$ is minimal.

The assertion about $H_1$ follows by abelianization.  Since $b_1(Z_V)=b_3(Z_V)=0$ rationally and $e(Z_V)=29$, one has $b_2(Z_V)=27$.  Poincar\'e duality and the universal coefficient sequence identify the torsion subgroup of $H_2(Z_V;\Z)$ with
\[
\operatorname{Ext}(H_1(Z_V;\Z),\Z)
\cong(\Z/2)^2.
\]
This proves the stated integral homology groups.
\end{proof}

\begin{corollary}\label{cor:klein-cover}
For a symplectically admissible gluing compatible with the dual surfaces, the universal fourfold cover of $Z_V$ is an exotic symplectic
\[
55\CP^2\#59\overline{\CP}^{\,2}.
\]
\end{corollary}

\begin{proof}
Join the punctured exceptional sphere on the negative side to the punctured genus-two surface $P$ on the positive side, as in Proposition~\ref{prop:w2}.  This gives a square $-1$ surface $S\subset Z_V$.  Its fundamental group maps trivially to $\pi_1(Z_V)$ because its map factors through the positive complement, whose image in the sum is trivial.  The inverse image of $S$ in the universal cover therefore consists of four disjoint copies, each of square $-1$.  Hence the universal cover is non-spin.

The universal cover is simply connected and has
\[
e=4e(Z_V)=116,
\qquad
\sigma=4\sigma(Z_V)=-4.
\]
Thus
\[
b_2=114,\qquad b_2^+=55,\qquad b_2^-=59.
\]
Its odd indefinite unimodular intersection form is
\[
55\langle1\rangle\oplus59\langle-1\rangle.
\]
Freedman's theorem gives the asserted homeomorphism.  The symplectic form on $Z_V$ pulls back to the universal cover, and Taubes's theorem gives a nonzero Seiberg--Witten invariant since $b_2^+=55>1$.  The standard connected sum has vanishing Seiberg--Witten invariants, so the two smooth structures are not diffeomorphic.
\end{proof}

\section{Exact computer verification}

The finite-index calculation uses exact Todd--Coxeter enumeration in finitely presented groups.

The exact computation files, recorded transcripts, software requirements,
and SHA-256 checksums are supplied with the paper in a fixed ancillary
archive:
\begin{center}
\small\path{CS_Torus_Surgery_Computations_v26.zip}.
\end{center}
An identical working copy is maintained at
\begin{center}
\url{https://github.com/akhmedov-coder/cartwright-steger-torus-surgery}.
\end{center}

The computation certifies three group-theoretic statements used in the
geometric argument.  It first verifies the index of the Cartwright--Steger
lattice $\Pi$ in the ambient CKY group $\Gamma$.  It then computes the index
of the subgroup $B$ generated by the normalization words and the node-joining
word, proving that $B$ is the mod-two kernel $N$.  Finally, it checks in each
of the three intermediate subgroups $\Pi_2$ that the normal closure of $B$
has index two.  The geometric identification of the entered word subgroup
with the image of the nodal curve is supplied by Lemma~\ref{lem:based-node};
the computer calculation is used only for these exact finite-index and
normal-closure assertions.

The ambient lattice is entered as
\[
\Gamma=\langle u,v,b\mid u^3,v^4,b^3,
(uv)^2(vu)^{-2},[v,b],(buv)^3,(buvu)^2v\rangle.
\]
Put $j=(uv)^2$ and define $a_1,a_2,a_3$ by the CKY words displayed in the ancillary files.  The subgroup
\[
\Pi=\langle a_1,a_2,a_3\rangle<\Gamma
\]
has index $864$.  This index is computed by the GAP script and reproduces the
CKY index, checking the transcription of the ambient presentation and the
three subgroup generators.

Next, the eight normalization words $p_1,\ldots,p_8$ displayed after
\cite[Proposition~2.4]{CKY} are entered as words in $a_1,a_2,a_3$ and $j$.
The two based branch paths from \cite[Section~2.3]{ASY} are
\[
r_1=a_3^3a_1a_2^{-1},\qquad
r_2=a_2a_1^{-2}a_3^{-1}a_1a_3^{-1}a_1^{-1}a_2^{-2},
\]
and the node-joining word is $h=r_1r_2^{-1}$.  Thus the subgroup supplied to the second enumeration is
\[
B=\langle p_1,\ldots,p_8,h\rangle<\Gamma.
\]

The ancillary file \texttt{CS\_curve\_subgroup\_index\_v5.py} uses SymPy 1.14.0.  It can be run by either of the commands
\begin{verbatim}
python3 CS_curve_subgroup_index_v5.py coset_table_based
python3 CS_curve_subgroup_index_v5.py relator_based
\end{verbatim}
The two strategies of the same SymPy Todd--Coxeter implementation construct
the same number of right cosets and terminate with
\begin{verbatim}
index_G_B 3456
index_Pi_B_if_contained 4 divisible True
\end{verbatim}
Consequently $[\Gamma:B]=3456$.  Together with the GAP calculation
$[\Gamma:\Pi]=864$, the index formula gives $[\Pi:B]=4$ once
$B\subset\Pi$ has been verified.

That containment is checked separately in GAP 4.12.1.  The command
\begin{verbatim}
iso := IsomorphismFpGroupByGenerators(H,[a1,a2,a3]);
curveWords := List([p1,p2,p3,p4,p5,p6,p7,p8,h],
                   x -> Image(iso,x));
\end{verbatim}
terminates for every listed word.  Here the domain of \texttt{iso} is
$H=\Pi<\Gamma$; thus each image is defined only after GAP has expressed the
corresponding word in the subgroup.  The words checked are
\[
p_1,\ldots,p_8,h.
\]
The GAP script also forms
\[
K=\langle\!\langle B\rangle\!\rangle_{\Pi}.
\]
It returns
\begin{verbatim}
[Pi:normal closure(B)] = 4
Size(Pi/normal closure(B)) = 4
IsAbelian = true
Abelian invariants = [ 2, 2 ]
\end{verbatim}
Thus $\Pi/K\cong(\Z/2)^2$.  Proposition~\ref{prop:ABN} gives
$B\subset N$ and hence $K\subset N$.  Since
$[\Pi:B]=[\Pi:K]=[\Pi:N]=4$, one obtains $B=K=N$.

Finally, the GAP script constructs the kernels $\Pi_2$ of the three nonzero
characters $\Pi/N\to\Z/2$.  For each character it verifies that all nine
generators of $B$ lie in $\Pi_2$ and that
\[
[\Pi_2:\langle\!\langle B\rangle\!\rangle_{\Pi_2}]=2.
\]
This gives the quotient used in Theorem~\ref{thm:group}.  In summary, GAP
computes $[\Gamma:\Pi]$ and the normal-closure quotients, while SymPy computes
$[\Gamma:B]$; Lemma~\ref{lem:based-node} identifies the entered word subgroup
with the geometric nodal-curve subgroup.

The ancillary files contain the presentations, words, commands, completion
counts, and recorded output for the SymPy and GAP calculations.  Their
SHA-256 hashes are
\begin{verbatim}
CS_curve_subgroup_index_v5.py:
  3aeb52b1a5833974a81687cc0e72d01ef376b8aaf474588e484ae032a1cc9402
CS_curve_subgroup_index_v5.transcript.txt:
  a62864706f9a8aad2373733826da3a28fa7c131d681fbc725d1a74e293bb3c07
CS_curve_subgroup_index_v5_relator.transcript.txt:
  c1b0fd4e22804f5bdf6c18c3cf057483afda5d6d180ae300caa78f6f48f91bc1
CS_curve_normal_closure_certificate_v6.g:
  af4c1a5ffd85d4895b45fa414a3958b6526be603c97b43cfbfcd1b863a48f0c3
CS_curve_normal_closure_certificate_v6.transcript.txt:
  f40924952486ddfd158a63e1cb6a31aa8afa01d4e0f784f3a199aa30cf47477d
\end{verbatim}

\bigskip
\begin{center}
 {\bf AI Use Disclosure} 
 \end{center}
 
 \medskip
 LLM tools were used for assistance for checking grammar and language editing. 
 All mathematical ideas, arguments, and results in this paper are due to the authors, who take full responsibility for the content.

\end{document}